\documentclass[12pt]{article}
\usepackage{amsmath,latexsym, amsthm, amsfonts, amssymb, amsxtra,amscd, enumerate,color, cancel}
\usepackage{bbm, dsfont,latexsym} 
\usepackage{float,wrapfig}
\usepackage[english]{babel}
\usepackage[usenames,dvipsnames,svgnames,table]{xcolor}
\usepackage{times, graphicx,stmaryrd}  
\RequirePackage[colorlinks,citecolor=blue,urlcolor=blue]{hyperref}
\usepackage[active]{srcltx}

\theoremstyle{plain}
\newtheorem{thm}{Theorem}[section]
\newtheorem{lemma}[thm]{Lemma}
\newtheorem{proposition}[thm]{Proposition}

\newtheorem{cor}[thm]{Corollary}

\theoremstyle{definition}

\newtheorem{remark}[thm]{Remark}

\allowdisplaybreaks
\newcommand{\less}{\hspace{-5em}}

\newcommand{\equa}{\begin{eqnarray*}}
\newcommand{\tion}{\end{eqnarray*}}
\newcommand{\equal}{\begin{eqnarray}}
\newcommand{\tionl}{\end{eqnarray}}

\newcommand{\R}{\mathbb{R}}
\newcommand{\N}{\mathbb{N}}
\newcommand{\E}{\mathbb{E}}
\newcommand{\PP}{\mathbb{P}}
\newcommand{\T}{\mathcal{T}}
\newcommand{\one}{\mathds{1}}
\newcommand{\half}{\frac{1}{2}}
\newcommand{\vare}{\varepsilon}
\newcommand{\om}{\omega}

\newcommand{\abs}[1]{\left|{#1} \right|}
\newcommand{\dist}{\textrm{dist}}
\newcommand{\set}[1]{\left \{#1 \right \}}
\newcommand{\Lap}{\mathcal{L}}
\newcommand{\pr}[1]{\left( #1 \right)}
\newcommand{\pl}[1]{\left[ #1 \right]}

\makeatletter
\def\timenow{\@tempcnta\time
\@tempcntb\@tempcnta
\divide\@tempcntb60
\ifnum10>\@tempcntb0\fi\number\@tempcntb
:\multiply\@tempcntb60
\advance\@tempcnta-\@tempcntb
\ifnum10>\@tempcnta0\fi\number\@tempcnta}
\makeatother

\title{On first exit times and their means for Brownian bridges}
\author{Christel Geiss$^1$   \hspace{0.7em}   Antti Luoto$^2$ \hspace{0.7em} Paavo Salminen$^3$\\ 
\today }
\date{}
\begin{document}

\maketitle
\begin{abstract}  For a Brownian bridge from $0$ to $y$  
we prove 
that the mean of the first exit time from interval $(-h,h), \,\, h>0,$ behaves as $O(h^2)$ when $h \downarrow 0.$
Similar behavior is seen to hold also for  the 3-dimensional  Bessel bridge. For Brownian bridge  and 
3-dimen\-sional  Bessel bridge  this mean of the first exit time has a puzzling  representation in terms of the Kolmogorov distribution. 
The result regarding the Brownian bridge  is  applied to prove in detail an estimate 
needed by Walsh to determine the convergence of the binomial tree scheme for European options. 
 \end{abstract}

\vspace{1em} 
{\noindent \textit{Keywords:}  Brownian motion, Brownian bridge,  Bes\-sel process, Bessel bridge,  first exit time, last exit time, Kolmogorov distribution 
function,  binomial tree scheme \\
\noindent AMS   60J65, 60J60; 91G60  
}
{\noindent
\footnotetext[1]{Department of Mathematics and Statistics, University of Jyv\"askyl\"a, Finland. \\ \hspace*{1.5em}
 {\tt christel.geiss{\rm@}jyu.fi}}
\footnotetext[2]{Department of Mathematics and Statistics, University of Jyv\"askyl\"a, Finland. \\ \hspace*{1.5em}  {\tt antti.k.luoto{\rm@}student.jyu.fi}}
\footnotetext[3]{Faculty of Science and Engineering, \r{A}bo Akademi University,  Finland. \\ \hspace*{1.5em}  {\tt paavo.salminen{\rm@}abo.fi}}


\section{Introduction}

We start with the description of our setting. Let $C[0,\infty)$ denote  the space of continuous functions  $\omega: [0,\infty)\mapsto \R,$ and 
$$  \mathcal{C}_t := \sigma\{ \om(s): s \le t\}$$
the smallest 
$\sigma$-algebra making  the co-ordinate mappings 
up to time $t$  measurable. Furthermore,  let
$ \mathcal{C}$ be  the smallest 
$\sigma$-algebra containing all   $\mathcal{C}_t,\, t\geq 0.$ For an interval $I\subset \R$  let $(\PP^X_x)_{x\in I}$ be a family of probability measures defined 
in the filtered canonical space  $(C[0,\infty),  \mathcal{C}, (\mathcal{C}_t)_{t\geq 0}) $ such that  under $\PP^X_x$ for a given $x\in I$  the co-ordinate process 
$X=(X_t)_{t \ge 0}:=(\omega_t)_{t \ge 0}$ is a regular diffusion  taking values in $I$ and starting from $x.$  Here, $X$ is considered in the  sense of It\^o and 
McKean \cite{ItoMcKean}, see also \cite{BorodSalm}.   A crucial property of $X$  is that there exists a (speed) measure $m^X$ such that the transition probability 
has a continuous strictly positive density $(t,x,y)\mapsto q_t(x,y),\, t>0, x,y\in I$ with respect to $m^X$ i.e.,
\begin{equation}
\label{de1}
\PP^X_x(X_t\in dy)= q_t(x,y) m^X(dy),
\end{equation}   
see  \cite[page 149 and 157]{ItoMcKean}.

For $(X_t)_{t\geq 0},\, X_0=x,$ as defined above and $T>0$, one can construct a new non-homogeneous strong Markov process by conditioning $X$ to be at a given point 
$y\in I$ at time $T.$  Although the conditioning is, in general, with respect to a zero set $\{X_T=y\}$, it can be realized using the Bayes formula and the notion 
of regular conditional distributions. Another approach is to apply the theory of the Doob $h$-transforms.  To  explain this briefly, consider $X$ in space-time i.e., 
the process $\bar X=((X_t,t))_{t\geq 0}.$  Introduce for $z\in I$ and $t<T$ the function 
$$
h(z,t):=h(z,t;y,T):= q_{T-t}(z,y).
  $$
By the Chapman-Kolmogorov equation it holds for $x\in I$ and $s<t$ 
\begin{eqnarray*}
&&\E^X_{(x,s)}[h(X_t,t)]=\int_I q_{t-s}(x,z)h(z,t)\,m(dz)\cr
&&\hskip2.7cm=
\int_Iq_{t-s}(x,z)q_{T-t}(z,y)\,m(dz)\cr 
&&\hskip2.7cm=q_{T-s}(x,y)\cr
&&\hskip2.7cm=h(x,s),
\end{eqnarray*}
where $\E^X_{(x,s)}$ refers to the expectation associated with the space-time process $\bar X$ initiated from $x$ at time $s.$ Consequently, we may define 
for $f\in{\cal B}_b(I)$(= bounded measurable functions on $I$) and $s<t<T$ a non-homogeneous Markov semigroup $P^{X,h}_{t,s},\, 0<s<t<T,$ via
\begin{equation}
\label{sg}
P^{X,h}_{t,s}f(x):=\E^X_x\Big[f(X_{t-s})\frac{h(X_{t-s},t)}{h(x,s)}\Big].
\end{equation}
The process governed by the probability measure induced by this semigroup is called the $X$-bridge to $y$ of length $T.$ The notations 
$X^{x,T,y},\ \PP^X_{x,T,y},$ and $\E^X_{x,T,y}$ are used for this process, its probability measure and the expectation, respectively,  when 
the initial state is $x\in I.$ From (\ref{sg}) one may deduce the following absolute continuity relation for $A_t\in \mathcal{C}_t,\ t<T,$
\begin{eqnarray*}
&&
\PP^X_{x,T,y}\big(A_t\big) =\E^X_x\Big[\frac{h(X_{t},t)}{h(x,0)}\,;\, A_t\Big] 
=\E^X_x\Big[\frac{q_{T-t}(X_{t},y)}{q_{T}(x,y)}\,;\, A_t\Big].
\end{eqnarray*}
We refer to Chung and Walsh \cite{ChungWalsh} for a general discussion on $h$-transforms, to  Fitzsimmons et al. \cite{Fitzsimmons}, in particular, 
Proposition 1,  for general Markovian bridges, and to Salminen \cite{Salm} for some properties of diffusion bridges. 

Our main interest in this paper is focused on the case where the underlying process is a Brownian motion. The notations $(W_t)_ {t\geq 0}$ and 
$(B^{x,T,y}_t)_{0\leq t <T}$ are used for standard Brownian motion and Brownian bridge from $x$ to $y$ of length $T,$ respectively, and, for notational 
simplicity,  the corresponding probability measures are  $\PP_x$ and    $\PP_{x,T,y}$ with the expectations $\E_x$ and    $\E_{x,T,y},$ respectively. 
Brownian bridge has in addition to the general $h$-transform approach  a few equivalent specific characterizations. Indeed, Brownian bridge  can be 
viewed as (i) a Gaussian process, (ii) a deterministic time change with an additional drift of standard Brownian motion or (iii) as a solution of a 
SDE, see e.g. \cite{BorodSalm} p. 66. In Section \ref{application} we apply, in particular,  (ii) and (iii). The second one states that 
\equal \label{bridge-representation}
B_t^{x,T,y}\stackrel{d}{=}  \left\{ \begin{array}{ll}
(1-\tfrac{t}{T})W\pr{\tfrac{Tt}{T-t}} + x + \frac{(y-x)t}{T}, & t < T\\
\ y & t = T,\\
\end{array} \right.
\tionl
where $\stackrel{d}{=} $ means that the processes on the left and the right hand side are identical in law. The third one says that  
\begin{align}\label{bridge-representation-II}
 B^{x,T,y}_t \stackrel{d}{=}   \left\{ \begin{array}{ll}
 (T-t)\int_0^t \frac{dW_s}{T-s}  + x +\frac{(y-x)t}{T},  & t < T\\ 
\ y & t = T.\\
\end{array} \right.
\end{align}
Here it is assumed that the canonical  filtration $({ \mathcal{C}}_t)_{t\ge0}$ is augmented with the null sets of $ \mathcal{C}$ with respect to $\PP_0$ 
in order to have the usual conditions satisfied.

For a general diffusion bridge, we present  an integral representation of the mean of the first exit time from an interval the aim being to deduce the limiting 
behavior of the mean,  when the length of the  interval around the starting value of the bridge decreases to $0.$ The main body of the paper concerns  Brownian 
bridge and Bessel bridge  (with dimension parameter 3).
For Brownian bridge (a similar result holds for Bessel bridge),  it is shown in Theorem \ref{brown-and-bessel} that
\equa 
\lim_{h \downarrow 0}  \frac{\E_{0,T,y}[\T_{(-h,h)}]}{h^2} =1, \quad y\neq 0,
\tion
where 
$$
\T_{(-h,h)}:= \inf \{ t > 0:  B^{0,T,y}_t \not \in (-h,h) \}
$$ 
denotes the first exit time from the interval $(-h,h),\, h>0.$  To avoid ambiguity, in some cases we indicate the process for which the first exit time is considered by 
writing for any continuous process $(X_t)_{t\geq 0}$ 
\[
 \T_{(a,b)}^X := \inf \{ t > 0:  X_t \not \in (a,b) \}, \quad  a < x < b,
\]
and put $\T_{(a,b)}^X :=\infty$ if $\{ t > 0:  X_t \not \in (a,b) \}= \emptyset.$
Recall  that for Brownian motion (this can be deduced, e.g., from the Laplace transform of $\T_{(a,b)}$ given in \cite[Part II, Section 1, 3.0.1]{BorodSalm}) it holds  
\equal \label{two-sided-Brownian}  
 \E_x \pl{\T_{(a,b)}}  = (b-x)(x-a),\quad a<x<b.
 \tionl
Consequently, 
\equa 
\lim_{h \downarrow 0}  \frac{\E_{0,T,y}[\T_{(-h,h)}]}{ \E_{0}[\T_{(-h,h)}]} =1,\quad y\neq 0.
\tion
For some other diffusion bridges a similar asymptotic behavior can be found. But we have not been able to prove such a result in generality.

\bigskip

 The paper is organized as follows. Section \ref{4} contains the main results.  In Subsection \ref{21}  integral representations are given for the 
mean of $\T_{(a,b)}^X$ when $X$ is, firstly, a general regular diffusion and, secondly, a corresponding diffusion bridge. We also calculate 
for a regular diffusion  $X,\ X_0=x,$  the limiting behavior of  the mean of  $\T_{(x-h,x+h)}^X$ as $h\to 0.$ In Subsection \ref{22} we focus on Brownian bridge 
and 3-dimensional Bessel bridge starting from $x>0$ and find the limiting behavior of  the mean of $\T_{(x-h,x+h)}$ as $h\to 0.$     
For the means of the first exit times -- considered in Subsection \ref{22} -- Subsection \ref{23}   provides   puzzling representations in terms of the 
Kolmogorov distribution function. We discuss also other properties of the Kolmogorov distribution, in particular, its connection with the last exit time 
distribution of Brownian motion. 
As an application, in Section \ref{application}, we use our results concerning Brownian bridge to give a detailed  proof of an estimate in  Walsh  \cite{Walsh} 
needed therein when  deriving  the convergence rate of an option price calculated from the binomial 
 tree scheme to the Black-Scholes price. The estimate is used also in a forthcoming paper \cite{Luoto}.  


\section{Main results} \label{4}

\subsection{Preliminaries}  \label{21}
Let $X=(X_t)_{t\geq 0}$ be a regular diffusion taking values on an interval $I,$ and recall from (\ref{de1}) the notation $q_t(x,y)$ for its transition density 
with respect to the speed measure $m^X.$  For $a<b,\, a,b\in I,$ let $\widehat X$ denote $X$  killed at $\T_{(a,b)}.$ Then $\widehat X$ is a regular diffusion on 
$(a,b).$ The speed measure of $\widehat X$ is $m^X,$ and $\widehat X$ has  a continuous strictly positive transition density $\widehat q_t(x,y)$ such that for $x,y\in(a,b)$
\begin{align} \label{speed-measure-relation}  
   {\PP}^{\widehat X}_x( \widehat X_t \in dy) &= \widehat q_t(x,y) m^X(dy) \notag\\
   &= {\PP}^{X}_x( X_t \in dy, \T_{(a,b)}>t) \notag  \\
   &= {\PP}^{X}_x(a < \inf_{0 \leq s \leq t} X_s, \sup_{0 \leq s \leq t}X_s < b, X_t \in dy).
\end{align} 

\noindent This yields immediately the following result.
\begin{proposition} 
In case  $\PP^X_{x}(\T_{(a,b)}<\infty)=1,$ i.e., $X$ does not die inside $(a,b),$ it holds
\begin{align}\label{representation-distributionII}
\E^X_{x}[\T_{(a,b)}] 
& = \int_{0}^\infty dt \int_{a}^{b} m^X(dz)\,  \widehat q_t(x,z). 
\end{align}
\end{proposition}
\noindent The next proposition will serve as an important tool  for the calculations below. 
\begin{proposition} 
For  $x \in (a,b) \subset I$ and $y \in I \backslash (a,b)$  it holds 
\begin{align}\label{representation-distribution}
\E^X_{x,T,y}[\T_{(a,b)}] 
& = \int_{0}^T dt \int_{a}^{b} m^X(dz)\,  \widehat q_t(x,z)  \frac{q_{T-t}(z,y)}{q_T(x,y)}.
\end{align}
\end{proposition}
\begin{proof} Since  $x\in (a,b)$ and $y \in I \backslash (a,b)$ we have $\PP^X_{x,T,y}(\T_{(a,b)}<T)=1.$ Consider
\begin{align*}
\E^X_{x,T,y}[\T_{(a,b)}] & = \int_{0}^T dt \int_{a}^{b}\PP^X_{x,T,y}(\T_{(a,b)}>t,\, X^{x,T,y}_t\in dz)\\
&=
\int_{0}^T dt \int_{a}^{b} {\PP}^X_{x}(\T_{(a,b)} > t,\, X_t\in dz)  \frac{q_{T-t}(z,y)}{q_T(x,y)}\nonumber\\
& = \int_{0}^T \int_{a}^{b} {{\PP}}^X_x(a < \inf_{0 \leq s \leq t} X_s,\, \sup_{0 \leq s \leq t}X_s < b, X_t \in dz) \frac{q_{T-t}(z,y)}{q_T(x,y)} dt,
\end{align*}
 where the Markov property and formula (\ref{sg}) are used.   

\end{proof}
One of our main issues concerns  the limiting behavior of  the mean of  $\T_{\!(x-h,x+h\!)}^X$ as $h\to 0$ for diffusion bridges.  For regular diffusions we have the following  fairly complete characterization. 
\begin{proposition}
Assume that the differential operator associated with the regular diffusion $X$ is given by
\begin{align*}
\mathcal{G} u(x):= \frac{1}{2}a^2(x) u^{\prime\prime}(x) + b(x) u^{\prime}(x), \quad \quad x \in I,
\end{align*}
where $x \mapsto a^2(x) > 0$ and $x \mapsto b(x)$ are continuous in $I$. Let $x_o \in \text{Int}(I)$. Then
\begin{align*}
\lim_{h \downarrow 0} \frac{\E_{x_o}^{X}[\mathcal{T}_{(x_o-h,x_o+h)}]}{h^2} = a^{-2}(x_o).
\end{align*}
\end{proposition}

\begin{proof}
Recall from \cite[Part I, Chapter II, No 7, p.17]{BorodSalm} that the speed measure $m^X$ and the scale function $s^X$ can be taken to be 
\begin{align}
m^{X}(dx)  = m^X(x) dx, \quad \frac{d}{d x}s^X(x)  = e^{-B(x)},  \label{speed_measure_and_scale_function_relations}
\end{align}
where 
\begin{align}
m^X(x)  = 2a^{-2}(x)e^{B(x)}, \quad B(x)  = \int^{x} 2a^{-2}(y) b(y) dy. \label{speed_density_and_B}
\end{align}
Consider now the process $X$ initiated at $x_o$ and killed when it leaves the interval $(x_o-h,x_o+h)$. We let $(\widehat{X}_t)_{t\geq0}$
denote this diffusion:
\begin{align*}
\widehat{X}_t = \left\{ \begin{array}{ll}
X_t, \quad & t < \mathcal{T}_{(x_o-h,x_o+h)},\\
\partial, \quad & t \geq  \mathcal{T}_{(x_o-h,x_o+h)},
\end{array}
\right.
\end{align*}
where $\partial$ is a cemetary point. It is well known (cf. \cite[Part I, Chapter II, No 11]{BorodSalm}) that the $0$-resolvent kernel of $\widehat{X}$ is given by 
\begin{align*}
\widehat{G}_0(x,y) \!:= \!\int_{0}^{\infty} \!\!\! \widehat{q}_t(x,y) dt = \left\{\begin{array}{ll}
\!\!\!\dfrac{(s(x)-s(x_o-h))(s(x_o+h)-s(y))}{s(x_o+h)-s(x_o-h)}, \,\, x < y,\\
&\\
\!\!\!\dfrac{(s(y)-s(x_o-h))(s(x_o+h)-s(x))}{s(x_o+h)-s(x_o-h)}, \,\, x > y,\\
\end{array}
\right.
\end{align*}
where $s := s^X$ and $\widehat{q}_t(x,y)$ is the transition density w.r.t. the speed measure $m^X$. Consequently, relation (\ref{representation-distributionII})
implies that
\begin{align}\label{mean_computation} 
\E^{X}_{x_o}[\mathcal{T}_{(x_o-h,x_o+h)}] 
   = & \int_0^{\infty} \left( \int_{x_o-h}^{x_o+h} \widehat{q}_t(x_o,y) m^X(dy) \right) dt\nonumber\\
   = & \int_{x_o-h}^{x_o+h} \widehat{G}_0(x_o,y) m^X(dy)\nonumber\\
   = & \,\frac{s(x_o{+}h)-s(x_o)}{s(x_o{+}h)-s(x_o{-}h)} \int_{x_o-h}^{x_o} \big(s(y) - s(x_o{-}h)\big)m^X(y) dy\nonumber\\
&  + \frac{s(x_o)-s(x_o{-}h)}{s(x_o{+}h)-s(x_o{-}h)} \int_{x_o}^{x_o+h} \big(s(x_o{+}h)-s(y)\big) m^X(y) dy.
\end{align}
Since $s$ is assumed to be continuously differentiable we have
\begin{align*}
\lim_{h \downarrow 0} \frac{s(x_o{+}h)-s(x_o)}{s(x_o{+}h)-s(x_o{-}h)} = \lim_{h \downarrow 0} \frac{s(x_o)-s(x_o{-}h)}{s(x_o{+}h)-s(x_o{-}h)} = \frac{1}{2},
\end{align*}
and l'Hospital's rule yields
\begin{align*}
& \lim_{h \downarrow 0} \frac{1}{h^2} \int_{x_o-h}^{x_o} (s(y) - s(x_o-h))m^X(y) dy\\
& = \lim_{h \downarrow 0} \frac{1}{2h} \int_{x_o-h}^{x_o} \left(-\frac{d}{dh} s(x_o-h) \right) m^X(y) dy\\
& = \lim_{h \downarrow 0} \left(-\frac{d}{dh} s(x_o-h) \right) \frac{1}{2h} \int_{x_o-h}^{x_o} m^X(y) dy\\
& = \frac{1}{2}s'(x_o)m^X(x_o)\\
& = a^{-2}(x_o),
\end{align*}
where the last equality follows from relations (\ref{speed_measure_and_scale_function_relations}) and (\ref{speed_density_and_B}). Similarly,
\begin{align*}
\lim_{h \downarrow 0} \frac{1}{h^2} \int_{x_o}^{x_o+h} \big(s(x_o+h)-s(y)\big)m^X(y) dy = a^{-2}(x_o).
\end{align*}
Hence, by (\ref{mean_computation}),
\begin{align*}
\lim_{h \downarrow 0} \frac{\E^X_{x_o}[\mathcal{T}_{(x_o-h,x_o+h)}]}{h^2} = a^{-2}(x_o).
\end{align*}
\end{proof}


\subsection{The mean of the first exit time for  Brownian bridge and 3-dimensional Bessel bridge} \label{22}

We introduce the following function
\equal \label{Delta}
\Delta(z,h,t):=\frac{1}{\sqrt{2\pi t}} \sum_{m=-\infty}^{\infty} \left (e^{-\frac{(z +4mh)^2}{2t}} - e^{-\frac{(z+2h(2m+1))^2}{2t}}  \right ), \,\,
|z| <h, \,t>0.
\tionl

\begin{thm}\label{brown-and-bessel} {\color{white} .} 
 \begin{enumerate}[(i)]
\item \label{bridge-mean}  For the Brownian bridge with $|y| \geq h,$
\begin{align}\label{mean-bridge-rep}
\E_{0,T,y}[\T_{(-h,h)}]  =  \int_{0}^T \int_{-h}^{h}  \frac{p_{T-t}(z,y)}{p_T(0,y)}\Delta(z,h,t) dz dt,
\end{align}
where $p_t(x,y)$ denotes the  transition density of the standard Brownian motion,
and,  moreover,
 \equal  \label{bridge-limit}
 \lim_{h\downarrow 0} h^{-2}\E_{0,T,y}[\T_{(-h,h)}] = 1, \quad y\neq 0.
 \tionl             
\item   \label{Bessel-mean} For the 3-dimensional Bessel bridge with  $x>h$ and  $y \notin (x-h,x+h),$
\equa
\E^{(3)}_{x,T,y}[\T_{(x-h,x+h)}]  
=  \int_{0}^T \int_{-h}^{h}   \frac{z+x}{x} \Delta(z,h,t)\frac{r^{(3)}_{T-t}(z+x,y)}{r^{(3)}_T(x,y)} dz dt,
\tion
where 
\equal \label{transition-Bessel3}
r^{(3)}_t(x,y) dy & = \frac{y}{x}\big(p_t(x,y) - p_t(x,-y) \big)dy, \quad x > 0, y>0 
\tionl
describes  the transition density of the 3-dimensional Bessel process, 
and, moreover, 
\equal \label{bessel-limit}
\lim_{h \downarrow 0} h^{-2} \E^{(3)}_{x,T,y}[\T_{(x-h,x+h)}]  = 1.
\tionl
\item  \label{infty-case}  
If $y \in (-h,h)$  (or $y \in (x-h,x+h)$ with $x>h,$ $y>0$)  the mean of the first exit time  is infinite for both bridges, i.e. 
\equal  \label{infinite-mean}
\E_{0,T,y}[\T_{(-h,h)}]  = \E^{(3)}_{x,T,y}[\T_{(x-h,x+h)}]  = \infty.
\tionl
\end{enumerate}
\end{thm}  
 
\begin{proof} \eqref{bridge-mean} 
According to \eqref{speed-measure-relation}  and  \eqref{representation-distribution}, 
   \begin{align*}
\E_{0,T,y}[\T_{(-h,h)}] \! =\!\!  \int_{0}^T\!\! \int_{-h}^{h} \!\!\frac{p_{T-t}(z,y)}{p_T(0,y)} \PP_0( \inf_{0 \leq s \leq t} W_s>-h,
\sup_{0 \leq s \leq t}W_s < h, W_t \in dz)
dt.
 \end{align*}
By \cite[Part II, Section 1,  1.15.8 (p. 180)]{BorodSalm} and \eqref{Delta}, 
\equa 
&& \PP_0( \inf_{0 \leq s \leq t} W_s>-h, \sup_{0 \leq s \leq t}W_s < h, W_t \in dz) 
 =\Delta(z,h,t) dz.
\tion 
To show \eqref{bridge-limit},  substitute $z=hu$ and $t=h^2s,$ and notice that $\Delta(hu,h,h^2 s) h  = \Delta(u,1, s),$ so that 
\equa 
\E_{0,T,y}[\T_{(-h,h)}] & = & \int_{0}^T \int_{-h}^{h}  \frac{p_{T-t}(z,y)}{p_T(0,y)}\Delta(z,h,t) dz dt  \notag\\
&=& h^2 \int_{0}^{T/h^2} \int_{-1}^{1}  \frac{p_{T-h^2s}(hu,y)}{p_T(0,y)}\Delta(u,1, s) du ds.
\tion
To apply dominated convergence for  $h\downarrow 0$ let $y$ be fixed and assume that  $h <\frac{|y|}{2}.$ 
Notice that for $y\neq0$ there exists a constant $C(T,y)>0$ such that
\equal \label{p-bound}
\sup_{0<t<T} p_t(0,y) \le C(T,y).
\tionl
This implies that  
\[
 \sup_{u \in (-1,1), \,  s \in (0,T/h^2)}  \frac{p_{T-h^2s}(hu,y)}{p_T(0,y)} \le   \frac{ C(T,y/2)  }{p_T(0,y)}.
 \]
Therefore, since $(u,s) \mapsto \Delta(u,1, s)$ is integrable and  $(u,s,h) \mapsto \frac{p_{T-h^2s}(hu,y)}{p_T(0,y)}$ is bounded
on $(-1,1)\times(0,T/h^2) \times (0,|y|/2),$   dominated convergence  yields
\equal \label{bridge-mean-limit}
\lim_{h \downarrow 0} \frac{\E_{0,T,y}[\T_{(-h,h)}]}{h^2} &=& \int_{0}^{\infty} \int_{-1}^{1} \Delta(u,1, s) du ds  \notag\\
&=&\int_{0}^{\infty} \int_{-1}^{1}  \PP_0(\T_{(-1,1)} >s, W_s \in du) ds  \notag \\
&=& \E_{0} [\T_{(-1,1)}] \notag \\
&=& 1,
\tionl
where \eqref{two-sided-Brownian} is used for the last equality. \\
\eqref{Bessel-mean} By \eqref{speed-measure-relation} and \eqref{representation-distribution}, the expectation  $\E^{(3)}_{x,T,y}[\T_{(x{-}h,x{+}h)}]$  has the representation  
\begin{align*} 
 \int_{0}^T \int_{x-h}^{x+h} \PP^{(3)}_x(\sup_{0 \leq s \leq t}\big|X_s - x\big| < h, X_t \in dz)   \frac{r^{(3)}_{T-t}(z,y)}{r^{(3)}_T(x,y)} dt.
\end{align*}
According to \cite[Part II, Section 5, 1.15.8]{BorodSalm} we have for $z \in (x-h,x+h)$ that
\begin{align*}  
& \PP^{(3)}_x( \sup_{0 \leq s \leq t} \abs{X_s - x} < h, X_t \in dz) \nonumber\\
& = \frac{z}{x\sqrt{2\pi t}} \sum_{m=-\infty}^{\infty} \pr{\exp\pr{-\tfrac{((z-x)+2h(2m))^2}{2t}} - \exp\pr{-\tfrac{((z-x)+2h(2m+1))^2}{2t}}} dz \nonumber\\
& = \frac{z}{x} \Delta(z-x,h,t) dz,
\end{align*}
with $\Delta$ defined in \eqref{Delta}.
We substitute $ z-x= \alpha$ and get
\begin{align} \label{Besselmean1exit}
\E^{(3)}_{x,T,y}[\T_{(x{-}h,x{+}h)}]   & =  \int_{0}^T \int_{x-h}^{x+h}   \frac{z}{x} \Delta(z-x,h,t)   \frac{r^{(3)}_{T-t}(z,y)}{r^{(3)}_T(x,y)} dz dt \notag\\
&=  \int_{0}^T \int_{-h}^{h}   \frac{\alpha+x}{x} \Delta(\alpha,h,t)   \frac{r^{(3)}_{T-t}(\alpha+x,y)}{r^{(3)}_T(x,y)} d\alpha dt.
\end{align}
For the proof of \eqref{bessel-limit} we substitute $t=h^2 s$ and $\alpha=h\beta$ so that
\begin{align*}
\E^{(3)}_{x,T,y}[\T_{(x{-}h,x{+}h)}]   & =  h^3 \int_{0}^{T/h^2} \!\!\!\!\int_{-1}^{1}   \frac{h\beta+ x}{x} \Delta(h\beta,h,h^2 s)   
\frac{r^{(3)}_{T-h^2 s}(h\beta+ x,y)}{r^{(3)}_T(x,y)} d\beta ds.
\end{align*}
Using relation \eqref{transition-Bessel3} yields
\equa
\left | \frac{h\beta+ x}{x} \, \frac{r^{(3)}_{T-h^2 s}(h\beta+ x,y)}{r^{(3)}_T(x,y)} \right | 
&=&  \left |  \frac{p_{T-h^2 s}(h\beta+ x,y)-p_{T-h^2 s}(h\beta+ x,-y)}{p_T(x,y)-p_T(x,-y)}\right |. \tion
To see that this expression is bounded in $s\in (0, T/h^2)$ and $\beta\in (-1,1)$ notice that $y \notin (h-x,h+x),$ with $x>h$ and $y>0$ 
implies
\[0<(|y-x|-h)^2 \le(h\beta+x-y)^2 \le (h\beta+x+y)^2,
\] 
 so that  
\[\left | p_{T-h^2 s}(h\beta+ x,y)-p_{T-h^2 s}(h\beta+ x,-y)\right |\le   p_{T-h^2 s}(|y-x|,h).\]
For $h \le |y-x|/2$ we  infer from  \eqref{p-bound} that 
\[ \sup_{s \in (0, T/h^2)} p_{T-h^2 s}(|y-x|,h) \le C(T, |y-x|/2).\]
Recall that $\Delta(\beta,1, s)=\Delta(h\beta,h,h^2 s)h$,  so that  dominated convergence gives
\begin{align*}
\lim_{h \downarrow 0} \frac{\E^{(3)}_{x,T,y}[\T_{(x{-}h,x{+}h)}]}{h^2} & = \int_{0}^{\infty} \int_{-1}^{1}   \Delta(\beta,1, s)  d\beta ds =1,
\end{align*} 
where the last equality was shown in \eqref{bridge-mean-limit} . \\
\eqref{infty-case} 
 For a regular diffusion $X,$   \eqref{infinite-mean} follows from
 $\PP^X_{x,T,y}(\mathcal{T}_{(x-h,x+h)} \ge T)>0.$  It holds
 \equal \label{limit-fraction}
 \PP^X_{x,T,y}(\T_{(x-h,x+h)} \ge T) = \lim_{\vare \downarrow 0 } \frac{ \PP^X_x(\T_{(x-h,x+h)} \ge T, y \le X_T \le y+\vare)}{ \PP^X_x(y \le X_T \le y+\vare)}
 \tionl 
and 
\equa
\PP^X_x(\T_{(x-h,x+h)} \ge T, y \le X_T \le y+\vare) = \int_y^{y+\vare} \widehat q_T(x,z) dz,
\tion  
where  $\widehat q$ is the transition density of $X$ killed at $\T_{(x-h,x+h)}.$ By \cite[p.157]{ItoMcKean} the transition density of
any regular diffusion is strictly positive. In our case this means
$$\widehat q_T(x,z)> 0 \quad \text{  for all  } \quad |z|<h.  $$
Using l'Hospital's rule in  \eqref{limit-fraction} yields
\equa
  \PP^X_{x,T,y}(\T_{(x-h,x+h)} \ge T) =  \frac{\widehat q_T(x,y)}{q_T(x,y)}>0.
\tion
\end{proof}

\begin{remark}{\color{white} .} 
 For the 3-dimensional Bessel bridge  starting in $0$ with $y >h$  it holds
 \begin{align*}
\E^{(3)}_{0,T,y}[\T_{h}] & = \int_{0}^T \int_{0}^{h} \PP^{(3)}_0(\sup_{0 \leq s \leq t} X_s < h, X_t \in dz) 
\frac{r^{(3)}_{T-t}(z,y)}{r^{(3)}_T(0,y)}dt,
\end{align*}
where $\T_{h}^X := \inf \{ t > 0:  X_t =h \}$ denotes the first hitting time.
 The represetation follows from  \eqref{speed-measure-relation} and \eqref{representation-distribution}, which can be extended to $x=0$
 since $0$ is an entrance boundary point. Similarly as above it can be shown that 
 \equa 
 \lim_{h \downarrow 0} \frac{\E^{(3)}_{0,T,y}[\T_{h}]}{h^2} =  \lim_{h \downarrow 0} \frac{\E^{(3)}_{0}[\T_{h}]}{h^2}   =\frac{1}{3}.
 \tion   
 We leave the proof to the reader. 

\end{remark}
\subsection{The Kolmogorov distribution function and the  mean of the  first exit time} \label{23}
  A classical result due to Doob \cite{doob49} is that   the distribution of  the supremum of the absolute value of the standard 
 Brownian bridge is given by 
  \equa 
 \PP_{0,1,0} \left(\sup_{0 \leq s \leq 1} \abs{B_s^{0,1,0}} \le h \right)= \sum_{m=-\infty}^{\infty} (-1)^m e^{-2 m^2 h^2}=:F(h).
 \tion 
 The function $h\mapsto F(h),\, h>0,$ is called the Kolmogorov distribution function due 
 to Kolmogorov's fundamental work \cite{kolmogorov33} (and also Smirnov \cite{smirnov39})  on empirical distributions. We refer also 
 to \cite{PitmanYor99} and     \cite[Section 5.7]{Kroese}. The main result of this section -- Theorem \ref{BigO2} -- provides   representations  for 
 the  mean of the  first exit time of the 
 Brownian bridge  and  of  the 3-dimensional Bessel bridge 
 involving the Kolmogorov distribution function. 
  We present now some formulas related to the Kolmogorov distribution  which we need later.  The following Jacobi's theta function indentity, an instance 
  of the Poisson summation formula, is stated in  \cite[equation (2.1)]{BianePitmanYor}):
 \begin{align*} 
 \sum_{m=-\infty}^{\infty} \cos(2m\pi v) e^{-m^2 \pi^2 u} = \frac{1}{\sqrt{\pi u}} \sum_{m=-\infty}^{\infty} e^{-\frac{(m+v)^2}{u}}, \quad u>0,v \in \R.
\end{align*}
 Putting here $u=2x^2/\pi^2$ and $v =1/2$ yields 
\equal \label{poisson-summation} 
F(x)=\sum_{m=-\infty}^{\infty} (-1)^m e^{-2 m^2 x^2}= \frac{\sqrt{2\pi}}{x} \sum_{k=1}^\infty \exp\left (-\frac{(2k-1)^2\pi^2}{8 x^2}\right ).
\tionl 
Notice also that
 \equa
 F(h/\sqrt{t}) &=&     \PP_{0,1,0} \left(\sup_{0 \leq s \leq 1} \abs{B_s^{0,1 ,0}} \le h/\sqrt{t} \right)\\
&=& \PP_{0,t,0} \left(  \sup_{0 \leq s \leq 1} \abs{ B_{ts}^{0,t,0}} \le h \right ) \\
&=&\PP_{0,t,0}\left (  \sup_{0 \leq s \leq t} \abs{B_s^{0,t ,0}} \le h\right ), 
 \tion
where it is used that 
$$\big(\sqrt t B_s^{0,1,0}\big)_ {0\leq s\leq 1} \stackrel{d}{=}  \big(\ B_{st}^{0,t,0}\big)_ {0\leq s\leq 1},$$ 
which can be seen by  applying  the scaling property of Brownian motion to the representation \eqref{bridge-representation}. 
Consequently, 
 \equal \label{Kolmogorov-relationII}
F(h/\sqrt{t}) =   \PP_{0,t,0} ( \T_{(-h,h)}>t )=\PP_{0,t,0} ( \T_{(-h,h)}=\infty ).
 \tionl
\begin{thm}\label{BigO2} {\color{white} .} \label{representation-with-Kolm}
\begin{enumerate}[(i)]
\item \label{mean-y-ge-h} For  the Brownian bridge with  $\abs{y} \geq h,$ 
\begin{align}\label{another-rep}
\E_{0,T,y}[\T_{(-h,h)}]  &= h \int_0^T \frac{ p_{T-t}(0,y)}{p_{T}(0,y)} \frac{F(h/\sqrt{t})}{\sqrt{2\pi t}} dt.
\end{align}
\item   \label{puzzleBessel} For the 3-dimensional Bessel bridge with positive $y \notin (x-h,x+h)$ and $x>0,$ 
\equal \label{puzzleBessel-rep}
\E^{(3)}_{x,T,y}[\T_{(x-h,x+h)}]  
= h \int_{0}^T \frac{r^{(3)}_{T-t}(x,y)}{r^{(3)}_T(x,y)}\frac{F(h/\sqrt{t})}{\sqrt{2\pi t}} dt.
\tionl
\end{enumerate}
\end{thm}  

\begin{proof}
\eqref{mean-y-ge-h} Since   $\E_{0,T,-y}[\T_{(-h,h)}] = \E_{0,T,y}[\T_{(-h,h)}],$  we may assume that $y \ge h.$ 
 We derive \eqref{another-rep} from \eqref{speed-measure-relation} and \eqref{representation-distribution} 
by showing that the Laplace transforms of the functions  
\equa 
T \mapsto \int_{0}^T \int_{-h}^h   p_{T-t}(x,y) \PP_{0,t,x} ( \T_{(-h,h)}>t )   p_t(0,x) dxdt
\tion 
and
\equal \label{function} 
T \mapsto h \int_{0}^T   p_{T-t}(0,y)  \frac{F(h/\sqrt{t})}{\sqrt{2\pi t}} dt
\tionl  coincide. In the following, we will denote the Laplace transform of a function $f$
by   $$\mathcal{L}_{t,\gamma}(f(t)) = \int_0^\infty e^{-\gamma  t} f(t) dt, \quad  \gamma >0.$$  Notice that  we also indicate the integration 
variable $t.$ 
Using Fubini's theorem and the  convolution formula,  we get
 \equal \label{the-conv-integral}
 && \less \Lap_{T,\gamma}\Big(\int_{0}^T \int_{-h}^h   p_{T-t}(x,y) \PP_{0,t,x} ( \T_{(-h,h)}>t )   p_t(0,x) dxdt\Big) \notag\\
 &=& \int_{-h}^h  \Lap_{t,\gamma}( p_{t}(x,y) )
   \Lap_{t,\gamma}(\PP_{0,t,x} ( \T_{(-h,h)}>t )   p_t(0,x) ) dx.
    \tionl  
To compute the second Laplace transform expression  of \eqref{the-conv-integral} we use the series representation 
\equal \label{general-Kolmogorov-relation}
 \PP_{0,t,x} ( \T_{(-h,h)}>t) =\sum_{m=-\infty}^\infty (-1)^m  e^{- \frac{2mh(mh-x)}{t}}, \quad |x|<h,
 \tionl
 (see  \cite[formula (4.12)]{Beghin} or \cite[formula (17)]{SalmYor}). 
For $x\in (-h,h)$ it holds that  
 \equa
\left | \sum_{|m| \ge 2} (-1)^m  e^{- \frac{2mh(mh-x)}{t}} \right | \le  \sum_{|m| \ge 2}  e^{- \frac{2mh(mh-x)}{t}} \le  2 \sum_{m \ge 2} \left( e^{- \frac{2h^2}{t}} \right)^m.
\tion
Hence one may  interchange the summation and the Laplace transform, and since 
\[\Lap_{t,\gamma}(p_{t}(x,z) ) =\frac{1}{\sqrt{2 \gamma}} e^{- \sqrt{2 \gamma}|z-x|},  \]         
one gets           
\equa
 && \less \Lap_{t,\gamma}(\PP_{0,t,x} ( \T_{(-h,h)}>t )   p_t(0,x) ) \\ 
 &=&  \Lap_{t,\gamma} 
  \left (\sum_{m=-\infty}^\infty (-1)^m  e^{- \frac{2mh(mh-x)}{t}} \frac{ e^{- \frac{x^2}{2t}}}{\sqrt{2 \pi t}} \right) \\
&=&    \sum_{m=-\infty}^\infty (-1)^m
   \Lap_{t,\gamma} \left ( p_{t}(x,2mh) \right)         \\
&=& \frac{1}{\sqrt{2\gamma }} \left (e^{-  \sqrt{2 \gamma} |x|}   + (e^{ \sqrt{2 \gamma}x}   +e^{ -\sqrt{2 \gamma}x})\sum_{m=1}^\infty (-1)^m
         e^{-  \sqrt{2 \gamma}2mh}  \right). \tion 
By a straightforward calculation, this yields
\equa  
&& \less \int_{-h}^h  \Lap_{t,\gamma}(p_{t}(x,y) )
  \Lap_{t,\gamma}( \PP_{0,t,x} ( \T_{(-h,h)}>t )   p_t(0,x) ) dx \notag \\
&=&   \frac{h}{2 \gamma} e^{- \sqrt{2 \gamma}y} \frac{1- e^{-\sqrt{2 \gamma} \,2h}}{ 1+ e^{-\sqrt{2 \gamma} \,2h}} \notag\\
   &=&  \frac{h}{2 \gamma} e^{- \sqrt{2 \gamma}y} \tanh( h\sqrt{2 \gamma}).
\tion 
We continue with the Laplace transform of \eqref{function}. From  \eqref{poisson-summation}  we get
\equa 
\Lap_{t,\gamma}\left ( \frac{F(h/\sqrt{t})}{h\sqrt{2\pi t}}\right )&=& \frac{1}{h^2} \sum_{k=1}^\infty \int_0^\infty \exp\left (-\frac{(2k-1)^2\pi^2 t}{8 h^2}\right )
\exp(-\gamma t)dt \notag  \\
&=&\sum_{k=1}^\infty\frac{8}{(2k-1)^2\pi^2 t + 8 h^2\gamma} \notag   \\
&=& \frac{ \tanh( h \sqrt{2 \gamma})}{h\sqrt{2\gamma}},
\tion 
where we  use that for any  $x \in \R$ it holds  
$$\tanh\left (\frac{\pi x}{2} \right) =  \frac{4x}{\pi} \sum_{k=1}^\infty\frac{1}{(2k-1)^2  + x^2}$$ 
(see  \cite[Subsection 1.421]{Gradshteyn}).  From  the  convolution formula 
we conclude that  for $y\ge h$ it holds
\equa 
  \Lap_{T,\gamma} \left (h \int_{0}^T   p_{T-t}(0,y)    \frac{F(h/\sqrt{t})}{\sqrt{2\pi t}}dt \right)  
&=&h \Lap_{t,\gamma} \left (p_t(0,y) \right)\Lap_{t,\gamma} \left (\frac{F(h/\sqrt{t})}{\sqrt{2\pi t}}\right)  \notag  \\
&=& \frac{h}{\sqrt{2 \gamma}} e^{- \sqrt{2 \gamma}y} \,\,   \frac{ \tanh( h\sqrt{2 \gamma})}{\sqrt{2 \gamma}} .
\tion 
This implies  \eqref{another-rep}, since we have shown that
\equal \label{equal-with-kolm}
&& \less 
\int_{0}^T \int_{-h}^h   p_{T-t}(x,y) \PP_{0,t,x} ( \T_{(-h,h)}>t )   p_t(0,x) dxdt \notag  \\
&& = h \int_{0}^T   p_{T-t}(0,y)  \frac{F(h/\sqrt{t})}{\sqrt{2\pi t}} dt.
\tionl 
\eqref{puzzleBessel} By \eqref{Besselmean1exit} it holds
 $$\E^{(3)}_{x,T,y}[\T_{(x-h,x+h)}]  
=  \int_{0}^T \int_{x-h}^{x+h}   \frac{z}{x} \Delta(z-x,h,t)   \frac{r^{(3)}_{T-t}(z,y)}{r^{(3)}_T(x,y)} dz dt. $$ 
We notice that $\Delta(z-x,h,t)$ given in \eqref{Delta} can be written as
\begin{align} 
\Delta(z-x,h,t) &= \frac{1}{\sqrt{2\pi t}} \sum_{m=-\infty}^{\infty} \left (e^{-\frac{(z-x+2h(2m))^2}{2t}} - e^{-\frac{(z-x+2h(2m+1))^2}{2t}} \right ) \nonumber\\
& = p_t(x,z)\sum_{m=-\infty}^{\infty}(-1)^m e^{-\frac{2mh(mh - (x-z))}{t}}  \nonumber\\
& =  p_t(x,z) \PP_{0,t,z{-}x}(\T_{(-h,h)} > t), \nonumber
\end{align}
where (\ref{general-Kolmogorov-relation}) is used for the last line. 
Since by \eqref{transition-Bessel3}  it holds 
\begin{align}\label{R3limit1}
\frac{r^{(3)}_{T-t}(z,y)}{r^{(3)}_T(x,y)} =  \frac{x}{z} \frac{p_{T-t}(z,y) - p_{T-t}(z,-y)}{p_T(x,y) - p_T(x,-y)},
\end{align}
we have
\begin{align} 
&\hspace{-1em}  (p_T(x,y)-p_T(x,-y))^{-1}  \E^{(3)}_{x,T,y}[\T_{(x{-}h,x{+}h)}]  \notag\\
& =\int_{0}^T \int_{x-h}^{x+h}  \PP_{0,t,z{-}x}(\T_{(-h,h)} > t)  p_{t}(x,z) (p_{T-t}(z,y) - p_{T-t}(z,-y)) dz dt\notag \\
& = \int_{0}^T\int_{-h}^{h}  \PP_{0,t,u}(\T_{(-h,h)} > t) p_t(0,u)p_{T-t}(u,y{-}x)du dt \nonumber\\
& \quad \quad -  \int_{0}^T \int_{-h}^{h} \PP_{0,t,u}(\T_{(-h,h)} > t) p_t(0,u)p_{T-t}(u,-(y{+}x))du dt  \nonumber\\
& =  h\int_{0}^T p_{T-t}(0,y{-}x) \frac{F(h/\sqrt{t})}{\sqrt{2\pi t}} dt- h\int_{0}^{T} p_{T-t}(0,{y{+}x}) \frac{F(h/\sqrt{t})}{\sqrt{2\pi t}} dt, \nonumber
\end{align}
where the last equality is implied by \eqref{equal-with-kolm}.
Then \eqref{R3limit1} yields
\begin{align*}
\E^{(3)}_{x,T,y}[\T_{(x{-}h,x{+}h)}] 
& = h \int_{0}^T \frac{r^{(3)}_{T-t}(x,y)}{r^{(3)}_T(x,y)}\frac{F(h/\sqrt{t})}{\sqrt{2\pi t}} dt.
\end{align*}
\end{proof}

\begin{remark} 
\begin{enumerate}[(i)]
\item We have not been able to find a probabilistic explanation for the appearance of the Kolmogorov distribution function in the representations 
\eqref{another-rep} and \eqref{puzzleBessel-rep}.
\item Notice that the  integrands w.r.t.~$t$ of  the expressions in  [Theorem \ref{brown-and-bessel}, equation \eqref{mean-bridge-rep}] and 
in  [Theorem \ref{representation-with-Kolm}, equation \eqref{another-rep}] do not coincide.
To see this, one can compare the double Laplace transform of   $$ hp_{T-t}(0,y) \frac{F(h/\sqrt{t})}{\sqrt{2\pi t}} \quad  \text {and }  \quad \int_{-h}^{h} p_{T-t}(z,y)\Delta(z,h,t) dz.$$ 
 Recall that $\frac{1}{\sqrt{2\pi t}}=  p_t(0,0),$
 and $F(h/\sqrt{t})=\PP_{0,t,0}(\T_{(-h,h)}>t)$ by \eqref{Kolmogorov-relationII}, so that, by a similar  computation as in the proof of  Theorem \ref{representation-with-Kolm},
 \equa
&& \less \Lap_{T, \gamma} \left (h \int_{0}^T   p_{T-t}(0,y) \frac{F(h/\sqrt{t})}{\sqrt{2\pi t}} e^{-\lambda t}dt \right)  \notag \\
&=&\Lap_{T, \gamma} \left (h \int_{0}^T   p_{T-t}(0,y)  \PP_{0,t,0}(\T_{(-h,h)}>t) p_t(0,0) e^{- \lambda t}dt \right)  \notag \\
&=& \frac{h}{\sqrt{2 \gamma}} e^{- \sqrt{2 \gamma}y} \,\,   \frac{ \tanh( h\sqrt{2 (\gamma+\lambda)})}{\sqrt{2 (\gamma+\lambda)}}.
\tion
On the other hand, we have  $\Delta(z,h,t)=   \PP_{0,t,z} ( \T_{(-h,h)}>t )p_t(0,z),$ which yields after some calculations that
\equa 
&& \less  \Lap_{T, \gamma} \left (\int_{0}^T \int_{-h}^h  p_{T-t}(z,y) \PP_{0,t,z}(\T_{(-h,h)}>t) p_t(0,z)dz e^{- \lambda t}dt \right)  \notag \\
&=& \frac{1}{ \lambda \sqrt{2 \gamma}} e^{- \sqrt{2 \gamma}|y|} \,\,  \left [ 1- \frac{ e^{- \sqrt{2 \gamma}h} + e^{\sqrt{2 \gamma}h} }{ e^{- \sqrt{ 2(\gamma+\lambda)}h} 
+ e^{\sqrt{ 2(\gamma+\lambda)}h} }\right ].
\tion
\end{enumerate}
\end{remark}

\bigskip
We conclude this section by pointing out a connection between the Kolmogo\-rov distribution function and the density of the last visit of $0$ by a Brownian motion before 
$\T_{(-h,h)}.$  This connection has been noticed by Knight in  \cite[Corollary 2.1]{Knight}.  
 
 We  provide here a different proof for this fact.

\begin{proposition} When  $\widehat W$ is a Brownian motion killed at $\mathcal{T}_{(-h,h)},$ then 
$t \mapsto \frac{F(h/\sqrt{t})}{ h \sqrt{2\pi t}}$  is the density of   the last passage time   $\lambda_0 := \sup \{t>0: \widehat W_t = 0\}.$ 
\end{proposition}
\begin{proof}
It is well known that (cf. \cite[Part I, Chapter II, No 20, p.~26]{BorodSalm}) 
\begin{align}\label{last-exit-density}
\PP_0^{\widehat W}(\lambda_0 \in dt) = \frac{\widehat q_t(0,0)}{\widehat G_0(0,0)}dt,
\end{align}
where 
$$\widehat q_t(x,y) = \frac{1}{\sqrt{2 \pi t}} \sum_{k=-\infty}^{\infty} \pr{e^{-\frac{(x-y+2k\cdot2h)^2}{2t}} - e^{-\frac{(x+y+(2k+1)2h)^2}{2t}}}$$
is the transition density (w.r.t.~the Lebesgue measure) of $\widehat W$  (cf. \cite[Part I, Appendix I, No 6, p.~126]{BorodSalm}), and
\begin{align*}
\widehat G_0(x,y) = \left\{ \begin{array}{ll}
 {\displaystyle\frac{(x+h)(h-y)}{h}}, & \quad -h \leq x \leq y \leq h,\\
&\\
 {\displaystyle \frac{(y+h)(h-x)}{h}}, & \quad -h \leq y \leq x \leq h,
\end{array}
\right.
\end{align*}
denotes the 0-resolvent kernel (see \cite[Part I, Appendix I, No 6, p.~126]{BorodSalm})).
From (\ref{last-exit-density}) we get 
 \begin{align*}
\PP_0^{\widehat W}(\lambda_0 \in dt)/dt & = \frac{1}{h \sqrt{2 \pi t}} \sum_{k=-\infty}^{\infty} \pr{e^{-\frac{4(2k)^2 h^2}{2t}} - e^{-\frac{4(2k+1)^2h^2}{2t}} }\\
& = \frac{1}{h\sqrt{2\pi t}} \pr{1 + 2\sum_{k=1}^{\infty}(-1)^k e^{-\frac{4k^2h^2}{2t}}}.
\end{align*}
\end{proof}


\section{Application} \label{application}
In this section we apply  our previous results  to prove in Corollary  \ref{corollary} an estimate  which was used by Walsh  in   \cite{Walsh}. The convergence analysis  there succeeds  to identify the leading constants appearing in the  error expansion  of 
$$\E  [g(X^{(n)}_n)- g(X_T)]$$ 
(cf. \cite[equation (14)]{Walsh})  
in terms of expressions depending on the  function $g$ (which is assumed to be exponentially  bounded and piecewise twice continuously differentiable). Here 
  $X_t = \sigma W_t$ for $t\ge 0$ and  $(X^{(n)}_k)_{k=0}^n$ denotes a symmetric simple random walk
with step size $ \sigma\sqrt{T/n}.$
For simplicity, we will put  $\sigma=1$  and hence consider  $$\E  [g(W^{(n)}_n) - g(W_T)].$$  The central idea for this error analysis is
to build  this random walk from a given  Brownian motion $(W_t)_{t \ge 0 }$ so that both processes are on the same probability space:
Fix $T>0$ and $n \in \N.$ For  $h:=\sqrt{T/n}$ 
 define $\tau_0 :=0$ and \equa   \tau_k := \inf\{ t> \tau_{k-1}: |W_t-W_{\tau_{k-1}}| = h \}, \quad k \ge 1.
\tion
Then  $(W_{\tau_k}   -W_{\tau_{k-1}})_{k=1}^\infty$ is a sequence of i.i.d. random variables  with  
$\PP(  W_{\tau_{k}}   -W_{\tau_{k-1}} =  h)$=  $\PP(  W_{\tau_{k}}   -W_{\tau_{k-1}} = - h)=     \frac{1}{2}.
$
Let  $k_*$ be such that   $\tau_{k_*}\le T< \tau_{k_*+1}.$
In \cite[Section 9]{Walsh} the conditional probability
\equa 
q(x) := \PP( k_* \text{ is even }| W_T = x), \quad x \in \R
\tion
\begin{wrapfigure}{r}{0.59\textwidth} 
  \vspace*{-3.7em} \hspace*{-2.6em}
     \includegraphics[width=0.71\textwidth]{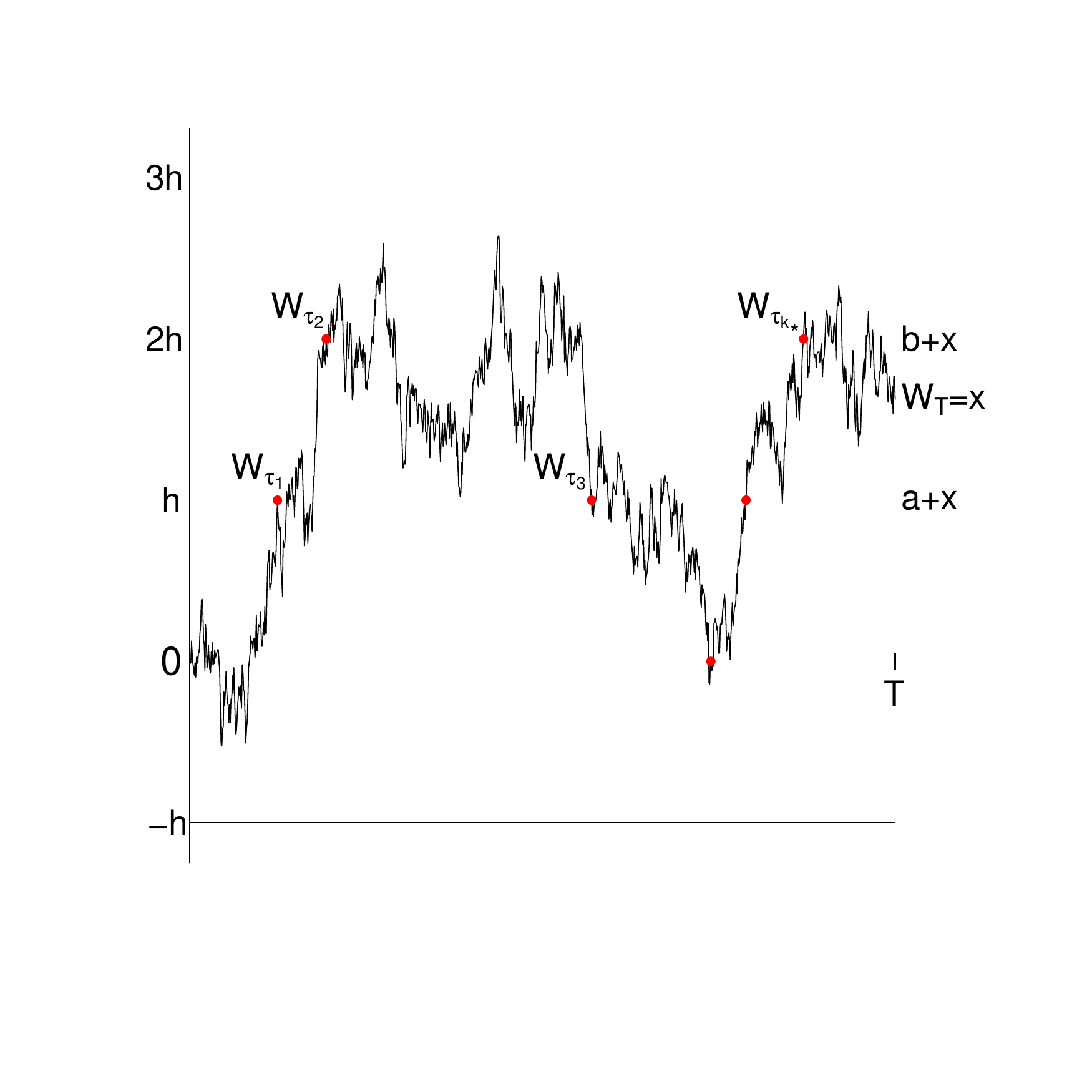}\vspace*{-7em}
     \end{wrapfigure}
is introduced (there $k_* $ is denoted by $L$). We want to study and estimate $q.$
Notice that $ k_*$ is an even number if and only if $W_{\tau_{k_*}}$ is an even multiple of $h.$ 
The process $(W_t)_{0 \le t \le T}$ given $W_0=0, W_T=x$ is
identical in law with a Brownian bridge from $0$ to $x$
and of length $T.$ We denote this bridge by $\left ( B^{0,T,x}_t\right)_{0 \le t \le T}$.
By time reversion, we get the Brownian bridge $\left ( B^{x,T,0}_t\right)_{0 \le t \le T}.$ 

 We fix $k \in \mathbb{Z}$ and  assume $x  \in ((2k-1)h, (2k+1)h).$  For simplicity  put  
\equal \label{h-points}
\underline{h} :=(2k-1)h, \quad   \overline{h} :=(2k+1)h,   \quad   h_{e}  := 2kh
\tionl 
for the lower and  upper value, and for 
 the 'even' midpoint of the interval. Then, given $W_T=x,$ we have that '$k_*$ is even' is the same as  '$W_{\tau_{k_*}}=  h_e$'.   
 For the time-reversed bridge associated to $\PP_{x,T,0}$  we get 
 \equa 
q(x)= \PP(k_* \text{ is even}| W_T = x)
 = \left\{ \begin{array}{ll}
\PP_{x,T,0}(\T_{h_e} < \T_{\underline{h}}), &  \underline{h} <x<  h_e,\\
\PP_{x,T,0}(\T_{h_e} < \T_{\overline{h}}), & h_e < x  < \overline{h},
\end{array}
\right.
\tion  
where $\T_y(\om) := \inf\{t>0: \omega(t)=y \}$ for $\omega \in C[0,T].$
 For the time-reversed and by $-x$  shifted bridge 
this means it hits the  'shifted even line'  $h_{e}-x$  before the  shifted odd one:
\equa 
 \PP(k_* \text{ is even}| W_T = x)
 = \left\{ \begin{array}{ll}
\PP_{0,T,-x}(\T_{h_e-x} < \T_{\underline{h} -x}), &  \underline{h} <x<  h_e,\\
\PP_{0,T,-x}(\T_{h_e-x} < \T_{\overline{h} -x}), & h_e < x  < \overline{h}.
\end{array}
\right.
\tion 
For any $a < 0 < b$ and $y \notin (a,b)$ it holds
\equa
\E \big[B^{0,T,y}_{\T_{(a,b)}}\big] &=& a  (1-\PP_{0,T,y}(\T_b < \T_a)) + b \, \PP_{0,T,y}(\T_b < \T_a).
 \tion
 Consequently, 
 \begin{align*}
\PP_{0,T,y}(\T_b < \T_a) = \frac{-a}{b-a} + \frac{\E \big[B^{0,T,y}_{\T_{(a,b)}}\big]}{b-a}.
\end{align*}
Hence
\begin{align}\label{q-equals}
q(x) = \left\{ \begin{array}{ll}
\dfrac{x-\underline{h} }{h} + 
 \dfrac{\E \big[B^{0,T,-x}_{\T_{( \underline{h} -x, h_e-x )}}\big]}{h}, \quad  \underline{h} <x<  h_e,\\
\dfrac{\overline{h} -x}{h} - 
 \dfrac{\E \big[B^{0,T,-x}_{\T_{(h_e-x,  \overline{h}-x )}}\big]}{h}, \quad h_e < x  < \overline{h}.
\end{array}
\right.
\end{align}
Arguing that  Brownian bridge and Brownian motion have  a similar exit behavior for small $h,$ in \cite[equation (20)]{Walsh} it is  stated that
\equal \label{walsh-O}
 q(x) = \frac{\dist(x,\N^h_o)}{h} + O(h) ,
 \tionl
where  $x\in \R,$
$\N^h_o = \set{(2k{+}1)h: k \in \mathbb{Z}}$   and  $\dist(x,\N^h_o) := \inf\{|x-y|: y\in \N^h_o \}.$
If one compares  \eqref{walsh-O} with \eqref{q-equals} one notices that 
$$\dist(x,\N^h_o)  \quad \text{is  equal to} \quad  x-\underline{h} \text{ or } \overline{h} -x,$$ so that we should have
$$ \E[ B^{0,T,-x}_{\T(x)}]=  O(h^2),$$
where  \eqref{h-points} was used to rewrite $\T_{( \underline{h} -x, h_e-x )}$ and $\T_{(h_e-x,  \overline{h}-x )}$ as
\equal \label{Tau-of-x} 
  \T(x):= \T_{(kh-x, (k+1)h-x)} \text{ if }  x\in (kh, (k+1)h).
 \tionl
In view of this we prove  the following lemma.
\begin{lemma}\label{bridgemeanlemma}
Suppose that $a < 0 < b$, $y \notin (a,b)$, $T > 0$ and $h > 0$. Then
\begin{align}\label{bridgemeanineq}
\big|\E \big[B^{0,T,y}_{\T_{(a,b)}}\big]\big|
\le \frac{\E_{0,T,y}[ \T_{(a ,b)}]}{T} \left(2(|a| \vee b) + |y| + 3\sqrt{2 T}\right).
\end{align}
\end{lemma}
\begin{proof}
By Markov's inequality we have 
\[ \left |\E \big [ B^{0,T,y}_{ \T_{(a ,b)}} \one_{\{\T_{(a ,b)} > T/2\}} \big] \right | 
\le (|a| \vee b) \PP_{0,T,y} (\T_{(a ,b)} > T/2)
\le \frac{2(|a| \vee b)}{T} \E_{0,T,y}[ \T_{(a ,b)}].
\] 
To estimate $ \left |\E \big[ B^{0,T,y}_{ \T_{(a ,b)}} \one_{\{\T_{(a ,b)} \le T/2\}}\big] \right |$ we let
$$\tilde{B}^{0,T,y}_t:= (T-t)\int_0^t \frac{dW_s}{T-s}  + \frac{t}{T}y, \quad t \in [0,T), $$
and $\tilde{B}^{0,T,y}_T :=y$. Then $\big( \tilde{B}_t^{0,T,y}\big)_ {0\leq t\leq T}     \stackrel{d}{=} \big( B_t^{0,T,y}\big)_ {0\leq t\leq T}  $ (cf. \eqref{bridge-representation-II}).
Setting
$$\tilde{\T}:= \inf\{t\in [0,T]:  \tilde{B}^{0,T,y}_t \notin (a,b) \}$$
yields $\E \big [\tilde{B}^{0,T,y}_{\tilde{\T}} \one_{\{\tilde{\T} \le  T/2\}}\big] =\E \big [ B^{0,T,y}_{ \T_{(a ,b)}} \one_{\{ \T_{(a ,b)} \le  T/2\}}\big] $ and
\equa
   \left |\E \big[\tilde{B}^{0,T,y}_{\tilde{\T}} \one_{\{\tilde{\T} \le  T/2\}}\big] \right |
   &\le&\!\!  T \left |\E \bigg[\one_{\{\tilde{\T} \le  T/2\}}\int_0^{\tilde{\T} \wedge(T/2)} (T-s)^{-1} dW_s \bigg]\right | \\
   && \!\!+  \E\left |(\tilde{\T} \wedge(T/2))  \int_0^{\tilde{\T} \wedge(T/2)} (T-s)^{-1} dW_s\right |  +  \frac{|y|}{T}\E  [\tilde{\T}] .
\tion
By the optional stopping theorem, H\"older's and Markov's inequality it holds that
\equa 
\left |\E \bigg[  \one_{\{\tilde{\T} \le  T/2\}}   \int_0^{\tilde{\T} \wedge\frac{T}{2}}\! (T-s)^{-1} dW_s   \bigg]\right | 
&\!=\!&\!\! \left |\E \bigg[\one_{\{\tilde{\T} >  T/2\}} \int_0^{\tilde{\T} \wedge \frac{T}{2}}(T-s)^{-1}  dW_s   \bigg] \right | \\
&\!\le \!&\!\! \left(\frac{2}{T}\E [\tilde{\T}] \right)^\half \left(\E \bigg[\int_0^{\tilde{\T} \wedge\frac{T}{2}}  \frac{ds}{(T/2)^2}\bigg] \right)^\half  \\
&\!\le \!&\!\! \left(\frac{2}{T} \E[\tilde{\T}] \right)^\half \left( \left( \frac{2}{T} \right)^\half \E \left( \tilde{ \T} \wedge \frac{T}{2} \right) \right)^\half\\
&\!\le\! & (\sqrt{2/T})^3 \, \E  [\tilde{\T}].
\tion
Again by H\"older's inequality, 
\equa 
\E\left |\Big(\tilde{\T} \wedge\frac{T}{2}\Big)  \int_0^{\tilde{\T} \wedge\frac{T}{2}}  \frac{dW_s}{T-s} \right |
&\le & \left (\E \left (\tilde{\T} \wedge \frac{T}{2} \right )^2 \right)^\half\left (\E \bigg[\int_0^{\tilde{\T} \wedge\frac{T}{2}}  \frac{ds}{(T-s)^2}\bigg] \right)^\half \\
& \le & \left( \frac{T}{2} \E[\tilde{\T}] \right)^{\frac{1}{2}} \left( \left( \frac{2}{T}\right)^\half \E \left(\tilde{T} \wedge \frac{T}{2} \right) \right)^\half\\
&\le & \sqrt{2/T} \, \E [\tilde{\T}].
\tion
From the above estimates and $\E_{0,T,y}[ \T_{(a ,b)}] =\E[\tilde{\T}]$  we get \eqref{bridgemeanineq}.
\end{proof}
Now we derive estimates for $\E_{0,T,y}[ \T_{(a ,b)}].$
\begin{lemma} \label{exit-time-estimate-bridge}
Let $T>0$ be fixed. Suppose that $a < 0 < b$ and $y \notin (a,b).$ 
\begin{enumerate}[(i)]
\item \label{small-y}  It holds 
\equal \label{tau-estimate}
\E_{0,T,y}\pl{\T_{(a, b)}} &\le & \left \{\begin{array}{lll}4b(|a|+y/2)\wedge T  \, \, &\text{ if }& \, \, y \ge b,  \\
                                            4 |a|(b+|y|/2)\wedge T &\text{ if }& \,\,y \le a. \end{array}   \right .
\tionl
\item \label{large-y} Define for 
$|y| \ge h$
\equal \label{C-and-meanfirstexit} 
C(T,h,y):= \frac{\E_{0,T,y} \big[\T_{(-h,h)}\big]}{h^2}.
\tionl
Then
\equa 
\E_{0,T,y} \big[\T_{(a,b)}\big] & \leq  C(b-a)^2 \quad \textrm{if } \abs{y} \geq b-a,
\tion 
where   $C = C(T, b-a, y),$  and it holds     
$$\lim_{b-a \to 0, \,  a<0<b}  C(T, b-a,y) = 1.$$
\end{enumerate}
\end{lemma}
\begin{proof}\eqref{small-y} We first assume that $y \ge b.$ Then
\begin{align*}
\T_{(a,b)}(B^{0,T,y}) &  \stackrel{d}{=} \inf \set{t \in [0,T): (1-\tfrac{t}{T})W_{\frac{Tt}{T-t}} + y\tfrac{t}{T} \notin (a,b)} \\
 & = \inf \set{t \in [0,T): (1-\tfrac{t}{T})W_{\tfrac{Tt}{T-t}}   \notin (a -y\tfrac{t}{T},b-y\tfrac{t}{T} )} \\
 & \le  \inf \set{t \in [0,T): (1-\tfrac{t}{T}) W_{\tfrac{Tt}{T-t}}  \notin (a -y\tfrac{t}{T},b )}  \\
& \leq \inf \set{t \in [0, T/2): W_{\tfrac{Tt}{T-t}} \notin   ( 2a-y, 2b)}.
\end{align*}
Since  $u:[0,T/2) \to  [0,T)$ given by $t \mapsto \frac{Tt}{T-t}$ is one-to-one, increasing and  $t =\frac{Tu(t)}{T+u(t)},$ we get 
\begin{align*}
 \inf \set{t \in [0, T/2)\!: W_{\frac{Tt}{T-t}}\notin\!(  2a-y, 2b)} 
& \!\leq\! \inf \set{u \in [0,T)\!:  W_u \notin(2a-y, 2b)}. 
\end{align*} Since by definition, 
$\T_{(a,b)}(B^{0,T,y}) \le T$ if $y \notin (a,b),$
we have 
\equa
\E_{0,T,y} [\T_{(a,b)}] &\leq& \E [\inf \set{u \in [0,T):  W_u \notin(  2a-y, 2b)}] \wedge T \\
&\le &\E_0 [ \T_{(2a-y,2b)}]\wedge T\\ 
&\leq& \abs{2a-y}2b \wedge T = 4b(\abs{a}+y/2)\wedge T.
\tion

The case $y \le a$ follows similarly.  \bigskip \\ 
\eqref{large-y} For $\abs{y} \geq b - a$  we have 
$\E_{0,T,y} \pl{\T_{(a,b)}} \leq \E_{0,T,y} \pl{\T_{(-(b-a), b-a)}}. $
Using \eqref{C-and-meanfirstexit} with $h=b-a$ gives for $|y| \ge h$ that
\equa
\E_{0,T,y} \pl{\T_{(-h,h)}} &=& h^2\, C(T, h,y), 
\tion
and from   Theorem \ref{brown-and-bessel} \eqref{bridge-mean}  we have that $C(T, h,y)$  converges to $1$  as $h \to 0.$ 
\end{proof}
 
Hence for    $x \in (kh, (k+1)h)$   and $k \in \{-1,0\}$ we get  by Lemma \ref{exit-time-estimate-bridge} 
\eqref{small-y} that  $\E_{0,T,-x}[\T_{(kh-x, (k+1)h-x)}]  \le  ch^2,$        
 and for  $k\not \in \{-1,0\}$   we use \eqref{large-y}.
Then Lemma \ref{bridgemeanlemma} implies $$\E[ B^{0,T,-x}_{\T(x)}]=  O(h^2).$$ 
 However, this equality   does not hold uniformly in $x$ with the consequence that  we can not use \eqref{walsh-O} in integrals like \eqref{q-integral} below. For example, for the 
sequence $x_k := (k+0.5)h$  it holds   
\equal \label{exit-limit}
\E[B^{0,T,-x_k}_{\T(x_k)}] = \E[ B^{0,T,-x_k}_{\T_{(-h/2, h/2)} }] \to -h/2, \quad  k\to \infty,
\tionl
which contradicts that $\E[ B^{0,T,-x}_{\T(x)}]=  O(h^2)$ holds uniformly in $x.$
The limit in \eqref{exit-limit} can be easily seen from the representation  $\big (B^{0,T,-x_k}_t \big)_{0 \le t \le T} \stackrel{d}{=} \big ( B^{0,T,0}_t -\frac{x_k t}{T} \big)_{0 \le t \le T}. $
For any path $t \mapsto B^{0,T,0}_t(\omega)$ one can find a sufficiently large $x_k,$ such that the transformed path  $t \mapsto B^{0,T,0}_t(\omega) -\frac{x_k t}{T} $  
exits $(-h/2, h/2)$ first at  $-h/2.$  

Nevertheless, one can prove the estimates needed in  \cite{Walsh}, where $q$ was used inside an integral over the real line.   We only discuss here  
\cite[equation (38)]{Walsh}, because the
calculations  for the other cases where the function $q$  appears are similar. For $\sigma = 1$  and denoting $\N^h_e := \set{2kh: k \in \mathbb{Z}}$ the last term 
in \cite[equation (38)]{Walsh} can be written as
\equal \label{q-integral}
&& \less \int_{-\infty}^{\infty}  \pr{2h^2 - \dist^2(x,\N^h_e)}q(x) p_T(0,x)dx  \notag\\
&= &\int_{-\infty}^{\infty}  \pr{2h^2 - \dist^2(x,\N^h_e)}  \dist(x,\N^h_o) h^{-1}   p_T(0,x)dx   \notag\\
&& + \int_{-\infty}^{\infty}  \pr{2h^2 - \dist^2(x,\N^h_e)}(q(x)- \dist(x,\N^h_o)h^{-1}  )  p_T(0,x)dx.
\tionl
The calculation   for the integral containing $\dist(x,\N^h_o)h^{-1} $ is carried out
in \cite{Walsh}. It remains to show that the other   integral behaves like $O(h^3).$ 
Since it holds  that $ \pr{2h^2 - \dist^2(x,\N^h_e)} \le 2 h^2$ and      
\[
|q(x)- \dist(x,\N^h_o)h^{-1} | = \big|\E \big[B^{0,T,-x}_{\T(x)}\big]\big| h^{-1}
\]
by \eqref{q-equals}  and \eqref{Tau-of-x}, we get the desired estimate from the next corollary.  
\begin{cor} \label{corollary}
For  $T > 0$ and  $h = \sqrt{T/n},$ there exists a $C = C(T) > 0$ such that
$$ \int_{-\infty}^{\infty} \big|\E \big[B^{0,T,-x}_{\T(x)}\big]\big| p_T(0,x) dx \leq C h^2,$$
where $\T(x)$ is  given in \eqref{Tau-of-x}.
\end{cor}

\begin{proof}
Since $B^{0,T,-x}_{\T(x)} \stackrel{d}{=} - B^{0,T,x}_{\T(-x)}$, it suffices to estimate the integral over $[0,\infty)$.
By \eqref{bridgemeanineq},
$$\big|\E \big[B^{0,T,-x}_{\T(x)}\big]\big| \le  \frac{\E_{0,T,-x}[ \T(x)]}{T} \left(2h + |x| + 3\sqrt{2 T}\right).  $$
If $x \in (0, h)$, then $\T(x) = \T_{(-x,h-x)}$, and estimate (\ref{tau-estimate}) gives 
\begin{align*}
\E_{0,T,-x}[\T(x)] \leq 4 x \pr{h -x+\frac{x}{2}} \leq 2h^2.
\end{align*}
For $x \ge h$ it holds  
$$\E_{0,T,-x}[\T(x)] \le C(T, h,-x) h^2$$
by Lemma \ref{exit-time-estimate-bridge}. From the above estimates we get 
\begin{align}\label{nollahoo}
\int_{0}^h \big|\E \big[B^{0,T,-x}_{\T(x)}\big]\big|\, p_T(0,x) dx
& \leq 2h^2 \int_0^h \frac{2h + x + 3\sqrt{2T}}{T} p_T(0,x) dx
\leq C(T)h^2,
\end{align}
and 
\equal \label{hooinfty}
 \int_{h}^{\infty}\!\! \big|\E \big[B^{0,T,-x}_{\T(x)}\big]\big| p_T(0,x) dx 
 &\le \!&\! h^2 \!\! \int_{h}^{\infty}\! \!C(T, h,-x)  \frac{2h + x + 3\sqrt{2T}}{T} p_T(0,x) dx\nonumber\\
& \le \!&\!  C(T) h^2 \int_{h}^{\infty}  (1+x)  C(T,h,-x) p_T(0,x) dx  \notag \\
& \le \!& \!C(T)h^2,
\tionl 
where the constant $C(T)$ varies from line to line.  
The last inequality in (\ref{hooinfty}) can be seen as follows. 
We use the representation \eqref{another-rep} for  $\eqref{C-and-meanfirstexit}$ and substitute  $t=h^2 u,$ so that
\begin{align*}
C(T,h,-x) =\E_{0,T,-x}[\T_{(-h,h)}] h^{-2}  &=h^{-1} \int_0^T \frac{ p_{T-t}(0,x)}{p_{T}(0,x)} \frac{F(h/\sqrt{t})}{\sqrt{2\pi t}} dt \\
&=  \int_0^{T/h^2} \frac{ p_{T-h^2u}(0,x)}{p_{T}(0,x)} \frac{F(1/\sqrt{u})}{\sqrt{2\pi u}} du.
\end{align*}
By Fubini's theorem it holds 
\begin{align*}
  \int_{0}^{\infty} \frac{F(1/\sqrt{u})}{\sqrt{2\pi u}} \int_{h}^{\infty} (1+x) p_{T-h^2u}(0,x)\one_{[0,T/h^2)}(u)dx du  \leq C(T),
\end{align*}
where we used for the last line that $u \mapsto \frac{F(1/\sqrt{u})}{\sqrt{2\pi u}} \one_{(0,\infty)}(u)$ is a density.
The claim then follows by (\ref{nollahoo}) and (\ref{hooinfty}).
\end{proof}
\bigskip


\bibliographystyle{plain}

\end{document}